\documentclass[12pt]{elsarticle}

\usepackage[left=1in, right=1in, top=1in, bottom=1in]
{geometry}
\usepackage{amsfonts,amssymb,graphicx,amsmath,amsthm, float}

\theoremstyle{plain}
\newtheorem{theorem}{Theorem}
\newtheorem{lemma}[theorem]{Lemma}
\newtheorem{prop}[theorem]{Proposition}
\newtheorem{rmk}[theorem]{Remark}
\newtheorem{definition}[theorem]{Definition}

\theoremstyle{remark}

\newcommand{\Z}{\mathbb{Z}}
\newcommand{\rn}{$rn$}

\def \diam {\text{diam}}

\begin{document}

\begin{frontmatter}

\title{The Radio numbers of all graphs of order $n$ and diameter $n-2$}
        \date{\today}

\author{Katherine Benson}
\address{The University of Iowa}
\address{katherine-f-benson@uiowa.edu}
\author{Matthew Porter}
\address{University of California, Santa Barbara}
\address{mattporter@math.ucsb.edu}
\author{Maggy Tomova (corresponding author)}
\address{The University of Iowa}
\address{maggy-tomova@uiowa.edu}
\address{319-335-0761}


\begin{abstract}
A radio labeling of a connected graph $G$ is a function $c:V(G) \to  \mathbb Z_+$ such that for every two distinct vertices $u$ and $v$ of $G$
$$\text{distance}(u,v)+|c(u)-c(v)|\geq 1+ \text{diameter}(G).$$
The radio number of a graph $G$ is the smallest integer $M$ for which there exists a labeling $c$ with $c(v)\leq M$ for all $v\in V(G)$.
The radio number of graphs of order $n$ and diameter $n-1$, i.e., paths, was determined in \cite{paths}. Here we determine the radio numbers of all graphs of order $n$ and diameter $n-2$.

\vspace{.1in}\noindent \textbf{2000 AMS Subject Classification:} 05C78 (05C15, 05C38)
\end{abstract}


\end{frontmatter}

\section{Introduction}

The general problem that inspired radio labeling is what has been known as the channel assignment problem:  the goal is to assign radio channels in a way so as to avoid interference between radio transmitters that are geographically close. The problem was first put into a graph theoretic context by Hale \cite{Georges}, \cite{Hale}.  The approach is to model the location of the transmitters via a graph with the transmitters corresponding to the vertices of the graph. Labels corresponding to the frequencies are assigned to the vertices so that vertices that are close to each other in the graph receive labels with large absolute difference. Exactly how large this absolute difference is varies in different models giving rise to several different kinds of labelings that are collectively known as $k$-radio labelings.

  Chartrand and Zhang were first to define $k$-radio labeling of graphs in \cite{Chartrand}.  Let $G$ be a connected graph and let $d(u,v)$ denote the distance between two distinct vertices $u$ and $v$ of $G$.  Let $\diam(G)$ be the diameter of $G$.  The $k$-radio labeling condition is that given $k$, with $1\leq k \leq diam(G)$, and $c$ a labeling, then $d(u,v) + |c(u)-c(v)| \geq 1 + k$ for all distinct vertices $u,v$ in $G$.  With this definition, one tries to minimize the largest value used as a label.
  
  It is important to note that $k$-radio labeling is actually a generalization of the classical idea of vertex coloring. Vertex coloring corresponds to $k$-radio labeling with $k=1$.

2-radio labeling is also known as $L(2,1)$ labeling. This is a labeling where adjacent vertices have labels with absolute difference at least $2$ and vertices distance 2 apart have labels with absolute difference at least $1$. This type of labeling was first studied by Griggs and Yeh in \cite{Griggs}. Since then a large body of literature has been compiled including the survey paper, \cite{Yeh}.

Another important specific $k$-radio labeling is when $k= diam(G)$. This is one of the most widely studied types of $k$-radio labelings and is known simply as a radio labeling, or multilevel distance labeling. This is the labeling that this paper will focus on so below are the relevant definitions.  


A \emph{radio labeling} $c$ of $G$ is an assignment of positive integers\footnote[1]{Some authors allow $0$ as a label. In this paper we do not allow $0$ to be a label and adjust all relevant formulas we cite accordingly} to the vertices of $G$ such that for every two distinct vertices $u$ and $v$ of $G$ the inequality
 $d(u,v)+|c(u)-c(v)|\geq \diam(G) + 1$, called the \emph{radio condition}, is satisfied. The maximum integer in the range of the labeling is its \emph{span}. The \emph{radio number} of $G$, $rn(G)$, is the minimum possible span over all radio labelings of $G$.

In \cite{paths} Liu and Zhu determine the radio number of graphs with $n$ vertices and diameter $n-1$, i.e., paths.  In this paper we determine the radio number of all graphs with $n$ vertices and diameter $n-2$.

Much of this paper will be devoted to studying a family of graphs which we call spire graphs. 

\begin{definition}Let $n,s \in \mathbb{Z}$ where  $n \geq 4$ and $2 \leq s \leq n-2$. 
The spire graph $S_{n,s}$ is the graph with vertices $\{v_1,...,v_n\}$, and edges $\{v_i,v_{i+1} | i=1,2,...,n-2\}$ together with the edge  $\{v_s,v_n\}$. The vertex $v_n$ is called {\em the spire}. Without loss of generality we will always assume that $s \leq \lfloor\frac{n}{2}\rfloor$. See Figure \ref{fig:SpireGraph}.
\end{definition}

\begin{figure}[h]
\begin{center}
\includegraphics[scale=.5]{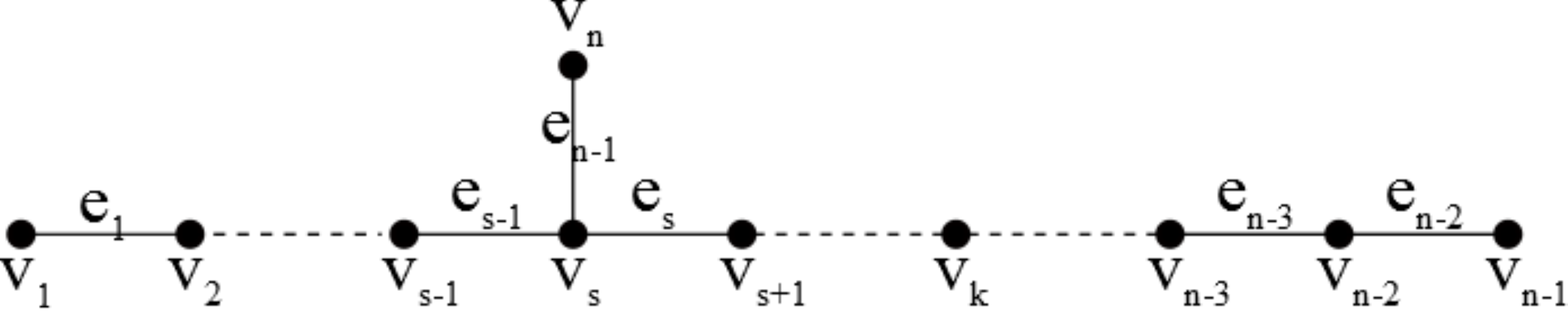}
\caption{$S_{n,s}$}
\label{fig:SpireGraph}
\end{center}
\end{figure}

We will show that:
 
 \medskip
{\bf Theorem (Radio Number of $S_{n,s}$}) 
{\em Let $S_{n,s}$ be a spire graph, where $2 \leq s \leq \lfloor\frac{n}{2}\rfloor$.  Then, }
\begin{center}

$$rn(S_{n,s}) =   \begin{cases}

 2k^2 -4k + 2s +3 & \text{ if } n = 2k \text{ and }2 \leq s \leq k-2,\\
2k^2 -2k   & \text{if } n = 2k \text{ and }s= k-1,\\
2k^2 -2k + 1  & \text{if } n = 2k \text{ and }s= k,\\
2k^2 -2k + 2s & \text{if } n = 2k + 1.

\end{cases}
$$
\end{center}

\medskip

Based on this result in Section \ref{sec:others} we will also determine the radio numbers of all other graphs with $n$ vertices and diameter $n-2$.

In her paper \cite{trees} Liu establishes bounds for the radio numbers of trees. In particular she determines the exact radio numbers of spire graphs with an odd number of vertices and of spire graphs when the spire is very close to the middle of the path. Although our techniques easily cover these cases as well, in the interest of brevity we will quote Liu's results whenever feasible.

The paper is structured as follows: in Section \ref{sec:upper}, we find an upper bound for the radio number of spire graphs by presenting algorithms for labeling them. Even though in some cases such upper bounds have been established in \cite{trees}, we nevertheless present our algorithms for all cases as these algorithms will be used again in Section \ref{sec:others}. In Section \ref{sec:techniques}, we establish techniques we will use in Section \ref{sec:lower} to find a lower bound for the radio number of spire graphs. When $n$ is odd, the lower bound follows directly from \cite{trees} so we simply quote the result. The same is true when $n$ is even and the spire is very close to the middle of the path. However, when $n$ is even and the spire is not near the middle of the path we introduce a number of new techniques that are likely to be applicable to other graphs that have large diameter. It is worth noting that these techniques can be used to easily reprove the lower bound for paths (both odd and even) found in \cite{paths}. In Section \ref{sec:others}, we analyze all remaining order $n$ graphs with diameter $n-2$. In most cases their radio number can be found by thinking of them as a spire graph with some additional edges. This section also has a number of lemmas that should be useful in other contexts.

\section{Radio number of spire graphs--upper bound}\label{sec:upper}

\begin{theorem}[{\bf Upper bound for $S_{n,s}$}]  \label{thm:upper}
Let $S_{n,s}$ be a spire graph, where $2 \leq s \leq \lfloor\frac{n}{2}\rfloor$.  Then, 
\begin{center}

$$rn(S_{n,s}) \leq   \begin{cases}

 2k^2 -4k + 2s +3 & \text{ if } n = 2k \text{ and }2 \leq s \leq k-2,\\
2k^2 -2k   & \text{if } n = 2k \text{ and }s= k-1,\\
2k^2 -2k + 1  & \text{if } n = 2k \text{ and }s= k,\\
2k^2 -2k + 2s & \text{if } n = 2k + 1.

\end{cases}
$$
\end{center}
\end{theorem}

\begin{proof}
To establish this bound we define a labeling with the appropriate span. The cases for $n$ even and $n$ odd are discussed separately.\\

{\bf Case 1:} First consider the case when $n = 2k$ for some $k\in\Z$.

{\em Subcase 1A:} $2 \leq s \leq k-2$ and $k \geq 7$.  Order the vertices of $S_{n,s}$ into three groups as follows:
\begin{center}
Group I: $v_k, v_{2k}, v_{k+4}, v_5, v_{k+3}, v_3, v_{k+2}, v_4,$\\
Group II: $v_{k+5}, v_6, v_{k+6}, v_7, \ldots, v_{k+m}, v_{m+1}, \ldots, v_{k + (k-3)}, v_{k-2},$\\
Group III: $v_{2k-2}, v_2, v_{k+1}, v_1, v_{2k-1}, v_{k-1}$.
\end{center}

In this ordering Group I always contains the same 8 vertices and Group III always contains the same 6 vertices. Group II follows the indicated pattern and contains $n-14$ vertices. 

Now, rename the vertices of $S_{n,s}$ in the above ordering by $x_1, x_2, \ldots, x_n$ where $x_1=v_k$, $x_2=v_{2k}$, etc. In Table \ref{Table 1} we define a labeling $c$ of $S_{n,s}$. We will let $c(x_1)=1$. The first column in the table gives the order in which the vertices are labeled, i.e., the inequality $c(x_i) > c(x_{i-1})$ always holds. The second column reminds the reader which vertex we are labeling. In the third column we have computed the distance between $x_i$ and $x_{i+1}$. Finally in the last column we give the difference between the labels $c(x_i)$ and $c(x_{i+1})$.  Given that $c(x_1)=1$, one can use the last column to compute $c(x_i)$ by summing the first $i-1$ entries of the column and then adding one to this sum.\\

{\bf Claim:} The function $c$ defined in Table \ref{Table 1} is a radio labeling on $S_{n,s}$.\\

\begin{table}

\centering
\begin{tabular}{|c|c|c|c|}
\hline
$x_i$ & Vertex Names & $d(x_i, x_{i+1})$ & $c(x_{i+1})-c(x_i)$\\
\hline
$x_1$ & $v_k$ & $k-s +1$ & $k+s-2$ \\
$x_2$ & $v_{2k}$ & $k-s+5$ & $k+s-6$\\
$x_3$ & $v_{k+4}$ & $k-1$ & $k$\\
$x_4$ & $v_5$ & $k-2$ & $k+1$\\
$x_5$ & $v_{k+3}$ & $k$ & $k-1$\\
$x_6$ & $v_3$ & $k-1$ & $k$\\
$x_7$ & $v_{k+2}$ & $k-2$ & $k+1$\\
$x_8$ & $v_4$ & $k+1$ & $k-2$\\
\hline
$x_9$ & $v_{k+5}$ & $k-1$ & $k$\\
$x_{10}$ & $v_6$ & $k$ & $k-1$\\
$\vdots$ & $\vdots$& $\vdots$ & $\vdots$\\
$x_{2m-1}$ & $v_{k+m}$ & $k-1$ & $k$\\
$x_{2m}$ & $v_{m+ 1}$ & $k$ & $k-1$\\
$\vdots$ & $\vdots$& $\vdots$ & $\vdots$\\
$x_{n-7}$ & $v_{k + (k-3)}$ & $k-1$ & $k$\\
$x_{n-6}$ & $v_{k-2}$ & $k$ & $k-1$\\
\hline
$x_{n-5}$ & $v_{2k-2}$ & $2k - 4$ & $4$\\
$x_{n-4}$ & $v_2$ & $k-1$ & $k$\\
$x_{n-3}$ & $v_{k+1}$ & $k$ & $k-1$\\
$x_{n-2}$ & $v_1$ & $2k-2$ & $2$\\
$x_{n-1}$ & $v_{2k-1}$ & $k$ & $k-1$\\
$x_n$ & $v_{k-1}$ & n/a & n/a\\
\hline
\end{tabular}
\caption{Hi}
\label{Table 1}
\end{table}
{\em Proof of claim:} To prove that $c$ is a radio labeling, we need to verify that the radio condition holds for all vertices $x_i, x_j \in V(S_{n,s})$.  In this case, the diameter of $S_{n,s}$ is $2k-2$ so we must show that for every $i,j$ with $j>i$, $d(x_i, x_j)+c(x_j)-c(x_i) \geq 2k-1$.\\

\indent {\em Case A:} $j=i+1$.  To verify the radio condition it suffices to add the entries in the $3^{rd}$ and $4^{th}$ columns of the $i^{th}$ row of Table \ref{Table 1} and check that this sum is always at least $2k-1$.
 
{\em Case B:} $j=i+2$.  Note that $c(x_j) - c(x_{i})$
is equal to the sum of the entries in the last column of rows $i$ and $i+1$ of Table \ref{Table 1}. One can quickly check that in most cases $c(x_j) - c(x_{i}) \geq 2k-2$ and therefore $d(x_i, x_j) + c(x_j) - c(x_{i}) \geq 1+ 2k-2$. 
It is less clear that the inequality $d(x_i, x_j) + c(x_j) - c(x_{i}) \geq 2k-1$ holds for the following six pairs of vertices: $\{x_3, x_1\}, \{x_4, x_2\},  \{x_{n-4}, x_{n-6}\},  \{x_{n-3}, x_{n-5}\},  \{x_{n-1}, x_{n-3}\}$, and $\{x_n, x_{n-2}\}$. In Table \ref{Table 2} we compute the distance between vertices and the difference between their labels for five of those vertex pairs. The reader can easily verify that these pairs satisfy the radio condition.
\begin{table}
\centering
\begin{tabular}{|c|c|c|}
\hline
Vertex pair & $d(x_{i}, x_{i+2})$ & $c(x_{i+2})-c(x_i)$\\
\hline
$\{x_3, x_1\}$ &  $4$ & $2k +2s -8$\\
$\{x_{n-4}, x_{n-6}\}$ &  $k-4$ & $k+3$\\
$\{x_{n-3}, x_{n-5}\}$ & $k-3$ & $k+4$\\
$\{x_{n-1}, x_{n-3}\}$  & $k-2$ & $k+1$\\
$\{x_n, x_{n-2}\}$ &  $k-2$ & $k+1$\\
\hline
\end{tabular}
\caption{}
\label{Table 2}
\end{table}
\\\indent For the pair $\{x_4, x_2\}$, note that the vertex incident to the spire is $v_{s}$.  We consider two cases:\\ 
\indent {\em(1)} If  $s < 5$, then $d(x_2, x_4) \,+\, c(x_4) \,-\, c(x_2) \,=\, d(v_n, v_5) \,+ \,2k \,+ \,s\,-\, 6 \,=\, 5 \,-\, s\,+\, 1 \,+\, 2k\,+\, s\,-\, 6 \,=\, 2k$.\\
\indent {\em(2)} If $s \geq 5$ then $c(x_4)\,-\,c(x_2) \,=\, 2k \,+\, s \,-\,6 \,\geq \,2k \,+\, 5\,-\, 6 \,=\, 2k \,-\, 1$.  

In both cases the radio condition is satisfied. 

 {\em Case C:} $j\geq i+3$.  Note that $c(x_j) - c(x_{i})$ is at least equal to the sum of the entries in the last column of rows $i$, $i+1$ and $i+2$ in Table \ref{Table 1}. As the sum of any three consecutive entries in the column is at least $2k-2$, in this case the radio condition is always satisfied.

\qed (Claim)

Letting $c(x_1)=1$, the largest number in the range of the radio labeling $c$ is  $c(x_n)$ and is therefore equal to the sum of the entries in the last column of Table \ref{Table 1} plus one.  Since the sums of Group I, Group II, and Group III are $8k+2s-9$, $(k-7)(2k-1)$, and $3k+4$, respectively, we conclude that $\rn(S_{n,s}) \leq 2k^2-4k+2s+3$ as desired.\\

{\em Subcase 1B:} $s=k-1$ and $k\geq 3$. As this algorithm is similar to the previous one but simpler, we summarize the algorithm directly in Table \ref{Table 3}.

\begin{table}
\centering
\begin{tabular}{|c|c|c|c|}
\hline
$x_i$ & Vertex Names & $d(x_i, x_{i+1})$ & $c(x_{i+1})-c(x_i)$\\
\hline
$x_1$ & $v_{k-1}$ & $k$ & $k-1$ \\
$x_2$ & $v_{2k-1}$ & $k+1$ & $k-1$\\
$x_3$ & $v_{2k}$ & $k$ & $k-1$\\
\hline
$x_4$ & $v_{2k-2}$ & $k$ & $k-1$\\
$x_5$ & $v_{k-2}$ & $k-1$ & $k$\\
$\vdots$ & $\vdots$& $\vdots$ & $\vdots$\\
$x_{2m}$ & $v_{2k-m}$ & $k$ & $k-1$\\
$x_{2m+1}$ & $v_{k-m}$ & $k-1$ & $k$\\
$\vdots$ & $\vdots$& $\vdots$ & $\vdots$\\
$x_{n-2}$ & $v_{2k-(k-1)}$ & $k$ & $k-1$\\
$x_{n-1}$ & $v_{k-(k-1)}$ & $k-1$ & $k$\\
\hline
$x_{n}$ & ${v_{k}}$ & n/a & n/a \\
\hline
\end{tabular}
\caption{}
\label{Table 3}
\end{table}
By adding the third and fourth entries in each row of Table \ref{Table 3}, we can verify that $d(x_i, x_{i+1})+c(x_{i+1})-c(x_i)\geq 2k-1$ for all $i$. In this case it is also easy to check that $c(x_{i+j})-c(x_{i})$ is at least $2k-2$ for all $i$ and all $j \geq 2$ so the radio condition is always satisfied. Adding one to the sum of the values in the last column of Table \ref{Table 3} gives the desired upper bound for the radio number in this case.\\
 
{\em Subcase 1C:} $s=k$ and $k\geq 2$.

Table \ref{Table 4} corresponds to the labeling algorithm. As in Subcase 1B checking that $c$ is a radio labeling is trivial. Again the sum of the values in the last column plus one gives the desired upper bound for the radio number.\\\\
\begin{table}
\centering
\begin{tabular}{|c|c|c|c|}
\hline
$x_i$ & Vertex Names & $d(x_i, x_{i+1})$ & $c(x_{i+1})-c(x_i)$\\
\hline
$x_1$ & $v_{k}$ & $k-1$ & $k$ \\
$x_2$ & $v_{1}$ & $k$ & $k-1$\\
$x_3$ & $v_{k+1}$ & $k-1$ & $k$\\
$\vdots$ & $\vdots$& $\vdots$ & $\vdots$\\
$x_{2m}$ & $v_{m}$ & $k$ & $k-1$\\
$x_{2m+1}$ & $v_{k+m}$ & $k-1$ & $k$\\
$\vdots$ & $\vdots$ & $\vdots$ & $\vdots$\\
$x_{n-2}$ & $v_{k-1}$ & $k$ & $k-1$\\
\hline
$x_{n-1}$ & $v_{2k-1}$ & $k$ & $k-1$\\
$x_{n}$ & ${v_{2k}}$ & n/a &  n/a \\
\hline
\end{tabular}
\caption{}
\label{Table 4}
\end{table}
{\bf Case 2:} Now suppose that $n = 2k+1$ for some $k\in\Z$. Order the vertices of $S_{n,s}$ as follows:
\begin{center}
Group I: $v_{k-1}, v_{2k-1}, v_{k-2}, v_{2k-2}, v_{k-3}, v_{2k-3}, \ldots, v_{k+3}, v_2, v_{k+2}$, \\
Group II: $v_{2k+1}, v_{k+1}, v_1, v_{2k}, v_k$.
\end{center} 

In this ordering Group I always contains $n-5$ vertices and Group II always contains the same 5 vertices. 
Now, rename the vertices of $S_{n,s}$ in the above ordering by $x_1, x_2, \ldots, x_n$.  This is the label order of the vertices of $S_{n,s}$.\\

\begin{table}
\centering
\begin{tabular}{|c|c|c|c|}
\hline
$x_i$ & Vertex Names & $d(x_i, x_{i+1})$ & $c(x_{i+1})-c(x_i)$\\
\hline
$x_1$ & $v_{k -1}$ & $k$ & $k$\\
$x_2$ & $v_{2k-1}$ & $k+1$ & $k-1$\\
$x_3$ & $v_{k-2}$ & $k$ & $k$\\
$x_4$ & $v_{2k-2}$ & $k+1$ & $k-1$\\
$x_5$ & $v_{k-3}$ & $k$ & $k$\\
$x_6$ & $v_{2k-3}$ &$k+1$ & $k-1$\\
$\vdots$ & $\vdots$ & $\vdots$ & $\vdots$\\
$x_{n-8}$ & $v_3$ & $k$ & $k$\\
$x_{n-7}$ & $v_{k+3}$ & $k+1$ & $k-1$\\
$x_{n-6}$ & $v_2$ & $k$ & $k$\\
$x_{n-5}$ & $v_{k+2}$ & $k+3- s$ & $k-3 + s$\\
\hline
$x_{n-4}$ & $v_{2k+1}$ & $k+2 - s$ & $k-2 + s$\\
$x_{n-3}$ & $v_{k+1}$ & $k$ & $k$\\
$x_{n-2}$ & $v_1$ & $2k-1$ & $1$\\
$x_{n-1}$ & $v_{2k}$ & $k$ & $k$\\
$x_n$ & $v_{k}$ & n/a & n/a\\
\hline
\end{tabular}
\caption{}
\label{Table 5}
\end{table}

{\bf Claim:} The function $c$ defined in Table \ref{Table 5} is a radio labeling on $S_{n,s}$.

\vspace{.3in}

{\em Proof of claim:} To prove that $c$ is a radio labeling, we need to verify that the radio condition holds for all vertices $x_i, x_j \in S_{n,s}$, i.e.,  we must show that for every $i,j$ with $j>i$, $d(x_i, x_j)+c(x_j)-c(x_i) \geq 2k$.

\indent {\em Case A:} $j=i+1$.  To verify the radio condition it suffices to add the entries in the $3^{rd}$ and $4^{th}$ column of the $i^{th}$ row of Table \ref{Table 5} and check that this sum is always at least $2k$.
 
 {\em Case B:} $j=i+2$.  Note that $c(x_j) - c(x_{i})$
is equal to the sum of the entries in the last column of rows $i$ and $i+1$ in Table \ref{Table 5}. One can quickly check that in most cases $c(x_j) - c(x_{i}) \geq 2k-1$ and therefore $d(x_i, x_j) + c(x_j) - c(x_{i}) \geq 1+ 2k-1=2k$.  It is less clear that $d(x_i, x_j) + c(x_j) - c(x_{i}) \geq 2k$ holds for the following five pairs of vertices $\{ u, v\}$: $\{x_{n-4}, x_{n-6}\}, \{x_{n-3}, x_{n-5}\},  \{x_{n-2}, x_{n-4}\},  \{x_{n-1}, x_{n-3}\}$, and $\{x_{n}, x_{n-2}\}$.  In Table \ref{Table 6} we compute the distance between vertices and difference between their labels for these vertex pairs.  The reader can verify that these pairs of vertices satisfy the radio condition keeping in mind that $s\geq 2$.

\begin{table}
\centering
\begin{tabular}{|c|c|c|}
\hline
Vertex pair & $d(x_{i}, x_{i+2})$ & $c(x_{i+2})-c(x_i)$\\
\hline
$\{x_{n-4}, x_{n-6}\}$ &  $s-1$ & $2k-3+s$\\
$\{x_{n-3}, x_{n-5}\}$ &  $1$ & $2k-5+2s$\\
$\{x_{n-2}, x_{n-4}\}$ & $s$ & $2k-2+s$\\
$\{x_{n-1}, x_{n-3}\}$  & $k-1$ & $k+1$\\
$\{x_n, x_{n-2}\}$ &  $k-1$ & $k+1$\\
\hline
\end{tabular}
\caption{}
\label{Table 6}
\end{table}

 {\em Case C:} $j\geq i+3$.  Note that $c(x_j) - c(x_{i})$ is at least equal to the sum of the entries in the last column of rows $i$, $i+1$ and $i+2$ in Table \ref{Table 5}. As the sum of any three consecutive entries in the column is at least $2k$, in this case the radio condition is always satisfied.

\qed (Claim)

The largest number in the range of the radio labeling $c$ is then $c(x_n)$ and is therefore equal to the sum of the entries in the last column of Table \ref{Table 5} plus one.  Since the sums of Group I and Group II are $(k-3)(2k-1) + 2k-3+s$ and $3k-1+s$, respectively, we conclude that $rn(G) \leq 2k^2 -2k +2s$ as desired.

\end{proof}

\section{Lower Bound Techniques} \label{sec:techniques}

In this section we develop some general techniques for determining a lower bound for the radio number of a graph. We will use these techniques to find a lower bound for the radio number of graphs with diameter $n-2$.

\begin{prop}\label{prop:untelescope}   Let $G$ be a connected graph with $n$ vertices and let $c$ be any radio labeling for $G$. Name the vertices of $G$ $\{x_1,...,x_n\}$ so that $c(x_i)<c(x_{i+1})$ for all $i$. For each $i$ let $j_i$ be a non-negative integer such that $d(x_i, x_{i+1})+c(x_{i+1})-c(x_i)=diam(G)+1+j_i$. Then 
\begin{center}
$c(x_n)=(n-1)(diam(G) + 1) + c(x_1) - \displaystyle \sum_{i=1}^{n-1} d(x_i,x_{i+1})+\displaystyle \sum_{i=1}^{n-1}j_i$.
\end{center}

\end{prop}

\begin{proof}
The result is easily obtained by adding up the equations
$$d(x_1, x_2) + c(x_2) - c(x_1) = diam(G) + 1+j_1,$$
$$d(x_2, x_3) + c(x_3) - c(x_2) = diam(G) + 1+j_2,$$
$$\vdots$$
$$d(x_{n-1}, x_n) + c(x_n) - c(x_{n-1}) = diam(G) + 1+j_{n-1}.$$

\end{proof}

\begin{lemma}\label{lemma:distance} [Maximum distance lower bound for $rn(G)$]
Let $G$ be a connected graph with $n$ vertices. Then
$$rn(G) \geq (n-1)(diam(G) + 1) + 1 - \max_p \displaystyle \sum_{i=1}^{n-1} d(x_i,x_{i+1}),$$ where the maximum is taken over all possible bijections $p$ from the vertices of $G$ to the set $\{x_1,...,x_n\}$.
\end{lemma}
\begin{proof} This result follows directly by minimizing the right side of the equation in Proposition \ref{prop:untelescope}.
\end{proof}

From Lemma \ref{lemma:distance} it is clear that finding $\max_p \sum_{i=1}^{n-1} d(x_i,x_{i+1})$ for a given graph will produce a lower bound for the radio number of this graph. This leads us to the following definition.

\begin{definition} We will call any labeling $c$ on a graph $G$ for which $\max_p \sum_{i=1}^{n-1} d(x_i,x_{i+1})$ is achieved a {\em distance maximizing labeling}. If $\max_p \sum_{i=1}^{n-1} d(x_i,x_{i+1}) - 1$ is achieved, we will call the labeling {\em almost distance maximizing}.
\end{definition}

The next lemma will be useful in finding the value for $\max_p \sum_{i=1}^{n-1} d(x_i,x_{i+1})$ for specific graphs.

\begin{lemma}\label{lem:distances}
Let $G$ be a graph with vertices $v_1,...,v_n$ and edges $e_1,...,e_m$. Let $p$ be a bijection from the vertices of $G$ to the set $\{x_1,...,x_n\}$. Let $P_j$ be a fixed shortest path from $x_j$ to $x_{j+1}$. Let $n(e_i)$ be the number of paths $P_j$ that contain the edge $e_i$. Then the following hold:
\begin{enumerate}

\item Each edge can appear in any path $P_j$ at most once.

\item Let $\{e^k_{i_1}, ... , e^k_{i_r}\}$ be all the edges incident to $x_k$. Then $n(e^k_{i_1})+ ... +n(e^k_{i_r})$ is even unless $k=1$ or $k=n$ in which case the sum is odd.

\item Suppose $e_i$ is an edge so that removing it from the graph gives a disconnected two component graph where the two components have $V_1$ and $V_2$ vertices, respectively.  Furthermore assume that if $x_j$ and $x_{j+1}$ are both contained in the same component, then so is $P_j$. Then $n(e_i)\leq 2 $min$\{V_1, V_2\}$.

\item Let $\{e_{i_1}, ... , e_{i_r}\}$ be a set of edges so that no two of them are ever contained in the same $P_j$. Then $n(e_{i_1})+ ... +n(e_{i_r})\leq n-1$.

\end{enumerate}
\end{lemma}

\begin{proof}
The first conclusion follows from the fact that $P_j$ is a shortest path so it cannot contain any loops.

The second conclusion follows from the fact that if $x_k$ is not the endpoint of a path $P_j$ but the vertex is included in this path, two of its incident edges belong to the path. If $x_k$ is the endpoint of a path, then exactly one of its incident edges is part of the path. For $1< k <n$, $x_k$ is the endpoint of exactly two paths while each of $x_1$ and $x_n$ is an endpoint of exactly one of the paths.

A path $P_j$ contains the edge $e_i$ if and only if its endpoints are in different components of the graph obtained by deleting $e_{i}$. This observation verifies the third conclusion.

The final conclusion follows from the fact that there are $n-1$ paths and any edge can appear in a path at most once.

\end{proof}

Sometimes we will need a generalization of the third condition of the lemma above, i.e., we will need to simultaneously remove multiple edges to disconnect a graph. The following lemma describes the corresponding result in this case. We will only need this more general version in the last section of the paper.
 
\begin{lemma}\label{lem:multiple}
Let $G$ be a graph with vertices $v_1,...,v_n$ and edges $e_1,...,e_m$. Let $p$ be a bijection from the vertices of $G$ to the set $\{x_1,...,x_n\}$. Let $P_j$ be a fixed shortest path from $x_j$ to $x_{j+1}$. Let $n(e_i)$ be the number of paths $P_j$ that contain the edge $e_i$.
Let $\{e_{i_1}, ... , e_{i_r}\}$ be a set of edges so that removing all of them from the graph gives a disconnected two component graph where the two components have $V_1$ and $V_2$ vertices, respectively. Furthermore assume that 

\begin{itemize}
\item If $x_j$ and $x_{j+1}$ are both contained in the same component, then so is $P_j$, and
\item Each path $P_j$ contains at most one of the edges $\{e_{i_1}, ... , e_{i_r}\}$. \end{itemize}

Then $n(e_{i_1})+ ... +n(e_{i_r})\leq 2 $min$\{V_1, V_2\}$.

\end{lemma}

\begin{proof}
 By the first condition a path $P_j$ can contain one of the edges $\{e_{i_1}, ... , e_{i_r}\}$ only if its endpoints are in different components of the disconnected graph. Thus there are at most $2 $min$\{V_1, V_2\}$ paths that contain one of these edges. By the second condition each path can contain at most one of the edges so $n(e_{i_1})+ ... +n(e_{i_r})\leq 2 $min$\{V_1, V_2\}$.
\end{proof}

\begin{rmk}
Let $G$ be a graph with vertices $v_1,...,v_n$ and edges $e_1,...,e_m$. Let $N(e_i)$ be the maximal value of $n(e_i)$ allowable under the conditions of Lemmas \ref{lem:distances} and \ref{lem:multiple}. Then  $\max_p \sum_{i=1}^{n-1} d(x_i,x_{i+1})\leq \sum_{j=1}^{m}N(e_j)$.
\end{rmk}

\section{Radio number of spire graphs--lower bound}\label{sec:lower}

We can now prove that the upper bound for $rn(S_{n,s})$ found in Section \ref{sec:upper} is also a lower bound. The result for odd values of $n$ follows easily from \cite{trees}. The proof for even values of $n$ is done in two steps. First we will compute a lower bound using Lemma \ref{lemma:distance} by determining $\max_p \sum_{i=1}^{n-1} d(x_i,x_{i+1})$ where $p$ is a bijection from $V(S_{n,s})$ to the set $\{x_1,...,x_n\}$. However this bound is not sharp so the second part of the proof shows how to improve the bound so it reaches the upper bound we established.

\begin{theorem} [\bf{Lower bound for $S_{n,s}$}]  \label{thm:lower}
Let $S_{n,s}$ be a spire graph, where $2 \leq s \leq \lfloor\frac{n}{2}\rfloor$.  Then, 
\begin{center}
$$rn(S_{n,s}) \geq   \begin{cases}

 2k^2 -4k + 2s +3 & \text{ if } n = 2k \text{ and }2 \leq s \leq k-2,\\
2k^2 -2k   & \text{if } n = 2k \text{ and }s= k-1,\\
2k^2 -2k + 1  & \text{if } n = 2k \text{ and }s= k,\\
2k^2 -2k + 2s & \text{if } n = 2k + 1.

\end{cases}
$$\end{center}
\end{theorem}

\begin{proof}

If $n = 2k+1$ the desired lower bound follows directly from Corollary 5 of \cite{trees}: we observe that $S_{n,s}$ is a spider (a tree with at most one vertex of degree more than two) so $$rn(S_{n,s}) \geq 2k^2 -2k + 2s.$$ Similarly if $n=2k$, and $s=k-1$ or $s=k$, the desired bound follows from Theorem 12 of \cite{trees}.

Assume then that $n=2k$, and $s \leq k-2$. First we determine $\max_p \sum_{i=1}^{n-1} d(x_i,x_{i+1})$ where $p$ is a bijection from $V(S_{n,s})$ to the set $\{x_1,...,x_n\}$.

Name the edges of $S_{2k,s}$ so that for $1\leq i \leq n-2$, $e_i$ is the edge between $v_i$ and $v_{i+1}$ and let $e_{n-1}$ be the edge between $v_s$ and $v_n$. The distance between $x_j$ and $x_{j+1}$ is the number of edges in the shortest path $P_j$ between these two vertices in the graph. Note that removing any edge $e_i$ from $S_{2k,s}$ results in a disconnected graph of two components. By the third and fourth conclusions of Lemma \ref{lem:distances}, (see also Figure \ref{fig:SpireGraph}), it follows that:

\begin{center}

$$N(e_i) =  \begin{cases}

2i & \text{if } i\leq s-1, \\
2i+2& \text{if } s\leq i\leq k-2, \\
2k-1& \text{if }  i= k-1,  \\
2(2k-1-i)& \text{if }  k\leq i\leq 2k-2,  \\
2& \text{if }   i= 2k-1.  \\
\end{cases}
$$
\end{center}

So $\max_p \sum_{i=1}^{n-1} d(x_i,x_{i+1})\leq \sum_{i=1}^{n-1} N(e_i)=2k^2 -2s + 1$. Thus we substitute this sum into the maximum distance lower bound to find that $$rn(S_{2k,s}) \geq 2k^2 - 4k + 2s + 1.$$

We now argue that this lower bound for $rn(S_{2k,s})$ can be increased by $2$. Recall that if $\tilde{c}$ is a radio labeling of $S_{2k,s}$ then for each $i \in \{1,...,n-1\}$ there is a non-negative integer $j_i$ such that $d(x_i, x_{i+1})+\tilde{c}(x_{i+1})-\tilde{c}({x_i})=n-1+j_i$. We will show that if $\tilde{c}$ is distance maximizing, then $\sum_{i=1}^{n-1}j_i\geq 2$ and if $\tilde{c}$ is almost distance maximizing, then $\sum_{i=1}^{n-1}j_i\geq 1$. In either case we conclude that $$rn(S_{2k,s}) \geq (2k^2 - 4k + 2s + 1)+2.$$

\vspace{.3in}

\textbf{Claim:} Let $c$ be a radio labeling of $G$ and let $\{x_{i-1}, x_i,  x_{i+1}\}$ be three consecutively labeled vertices such that $c(x_{i-1})<c( x_i)<c( x_{i+1})$. Assume that $x_{i-1}, x_{i+1} \in \{v_1, v_2,..., v_s, ... v_{k-1}, v_{n}\}$ and $x_i \in \{v_k, v_{k+1}, ..., v_{2k-1}\}$.  Let $\alpha$ denote $x_{i-1}$ or $x_{i+1}$, whichever has smaller distance to $x_i$, and we let $\beta$ denote the one with the larger distance to $x_i$ (the only case in which the two distances are equal is when $x_{i-1}=v_n$ and $x_{i+1}=v_{s-1}$ (or vice versa); in this case let $\alpha$ be $v_n$). Let $j_i$ and $j_{i+1}$ be non-negative integers such that $$d(x_i,  \alpha) + |c(x_i) - c( \alpha)| = n-1+ j_i$$ and $$ d(x_i, \beta) + |c(\beta) - c(x_i)| = n-1+ j_{i+1}.$$ Then $$j_i + j_{i+1} \geq \left\{
\begin{array}{c l}
  2(d(x_i, \alpha)) - n + 1 & \alpha \neq\ v_n, \\
  2(d(x_i, \alpha)) - n - 1 & \alpha = v_n,
\end{array}
\right.
$$

{\em Proof of claim:}
Let $\{x_{i-1}, x_i, x_{i+1}\}$ be a triple of vertices satisfying the hypotheses of the claim.  We observe that $$d(\alpha,\beta) = \left\{
\begin{array}{c l}
  d(x_i, \beta) - d(x_i, \alpha) & \alpha\neq\ v_n, \\
  d(x_i, \beta) - d(x_i, \alpha) + 2 & \alpha = v_n.
\end{array}
\right.
$$ 

We will prove the claim in detail in the case when $c(\alpha) < c(x_i) < c(\beta)$ and $\alpha \neq v_n$. For the other cases we only present the final result and let the interested reader verify the details of the computations.  

The radio condition applied to the pair of vertices $\alpha$ and $\beta$ gives $$n-1 \leq d(\alpha,\beta) + c(\beta) - c(\alpha).$$ 
We substitute $d(\alpha,\beta) =  d(x_i, \beta) - d(x_i, \alpha)$ in the above equation and add and subtract $c(x_i)$ to obtain $$n-1 \leq d(x_i, \beta) - d(x_i, \alpha) + c(\beta) - c(\alpha) + c(x_i) - c(x_i).$$ Recall that
$$d(x_i,  \alpha) + c(x_i) - c( \alpha) = n-1+ j_i$$ and $$ d(x_i, \beta) + c(\beta) - c(x_i) = n-1+ j_{i+1}$$ for some non-negative integers $j_i$ and $j_{i+1}$.
We now make a series of substitutions to obtain a lower bound for $j_i + j_{i+1}$.  First, we substitute $d(x_i, \beta) + c(\beta) - c(x_i) = n-1+j_{i+1}$ and add and subtract $j_i$ to obtain $$n-1 \leq n - 1 + j_{i+1} - d(x_i, \alpha) - c(\alpha) + c(x_i) + j_i - j_i.$$
Now, we substitute $n-1 + j_i = d(x_i, \alpha) + c(x_i) - c(\alpha)$, which yields, after cancelling $d(x_i, \alpha)$, $$n-1 \leq 2(c(x_i) - c(\alpha)) + j_{i+1} - j_i.$$
Solving for $c(x_i) - c(\alpha)$ and multiplying through by $(-1)$ shows that 
$$c(\alpha) - c(x_i) \leq \frac{1}{2}(-n + 1 + j_{i+1} - j_i).$$ 
Then
$$d(x_i, \alpha) + c(x_i) - c(\alpha) = n-1+ j_i$$
$$ \implies d(x_i, \alpha) = n-1+j_i +c(\alpha) - c(x_i) $$
$$ \implies d(x_i, \alpha) \leq n-1+j_i +\frac{1}{2}(-n + 1 + j_{i+1} - j_i) = \frac{1}{2}(n-1+j_i + j_{i+1})$$
$$ \implies j_i + j_{i+1} \geq 2(d(x_i, \alpha)) - n + 1,$$
and we have obtained the desired lower bound for $j_i + j_{i+1}$.  Making similar series of substitutions in the other three cases depending on the label order of $\alpha, x_i$ and $\beta$ and on whether or not $\alpha = v_n$ shows that $$j_i + j_{i+1} \geq \left\{
\begin{array}{c l}
  2(d(x_i, \alpha)) - n + 1 & \alpha \neq\ v_n, \\
  2(d(x_i, \alpha)) - n - 1 & \alpha = v_n.
\end{array}
\right.
$$

\qed

\vspace{0.3in}

From these two inequalities, we construct Table \ref{Table 7}, in which each entry gives the lower bound for the $j_i + j_{i+1}$ associated to the corresponding $x_i \in \{v_k, ..., v_{2k-1}\}$ and $\alpha \in \{v_1, ..., v_{k-1}, v_{n}\}$ based on the equation above.

\begin{table}
\centering
    \begin{tabular}{ | c | c | c | c | c | c | c | c | }
    \hline
      & $v_k$ & $v_{k+1}$ & $v_{k+2}$ & $...$ & $v_{2k-3}$ & $v_{2k-2}$ & $v_{2k-1}$ \\\hline
     $v_1$ & $0$ & $1$ & $3$ & $...$ & $2k-7$ & $2k-5$ & $2k-3$ \\\hline
     $v_2$ & $0$ & $0$ & $1$ & $...$ & $2k-9$ & $2k-7$ & $2k-5$ \\\hline
     $v_3$ & $0$ & $0$ & $0$ & $...$ & $2k-11$ & $2k-9$ & $2k-7$ \\\hline
     $\vdots$ & $\vdots$ & $\vdots$ & $\vdots$ & $...$ & $\vdots$ & $\vdots$ & $\vdots$ \\\hline
     $v_{k-3}$ & $0$ & $0$ & $0$ & $...$ & $1$ & $3$ & $5$ \\\hline
     $v_{k-2}$ & $0$ & $0$ & $0$ & $...$ & $0$ & $1$ & $3$ \\\hline
     $v_{k-1}$ & $0$ & $0$ & $0$ & $...$ & $0$ & $0$ & $1$ \\\hline
     $v_n$ & $\geq 0$ & $\geq 0$ & $\geq 0$ & $...$ & $\geq 0$ & $\geq 1$ & $\geq 3$ \\\hline
    \end{tabular}
\caption{}
\label{Table 7}
\end{table}

Suppose $c$ is any distance maximizing radio labeling of $S_{2k,s}$. Note that in this case $n(e_{k-1})=2k-1$ so by conclusion 3 of Lemma \ref{lem:distances} if $x_{i}$ is in the set $\{v_k, ..., v_{2k-1}\}$, then $x_{i-1}$ and $x_{i+1}$ are in the set $\{v_1, ..., v_{k-1}, v_n\}$ so the hypotheses of the claim are satisfied for the triple $\{x_{i-1}, x_i, x_{i+1}\}$. By the claim a lower bound for $j_i + j_{i+1}$ is given by Table \ref{Table 7}. Let $m$ be such that $x_m=v_{2k-1}$. In any distance maximizing radio labeling $n(e_{2k-2})=2$. By conclusion 2 of Lemma \ref{lem:distances}, as $n(e_{2k-2})$ is even, $v_{2k-1}$ is not the first or last labeled vertex. Therefore $1<m<n$ and we can use Table \ref{Table 7} to compute a lower bound of 1 for $j_m+j_{m+1}$.

 If $j_m+j_{m+1}>1$ then $\sum_{i=1}^{n-1} j_i\geq 2$ as desired. If $j_m+j_{m+1}=1$ then either $x_{m-1}$ or $x_{m+1}$, whichever is closest to $v_{2k-1}$, is $v_{k-1}$, as this is the only row with an entry less than 2 in the last column of Table \ref{Table 7}. In any distance maximizing labeling, $v_{k-1}$ must be the first or last vertex labeled because $n(e_{k-2})+n(e_{k-1})$ is odd. Without loss of generality assume that $v_{k-1}$ is the first labeled vertex and so $m=2$. Now consider the vertex $v_{2k-2}$ which corresponds to some $x_r$ with $r\geq 4$. Therefore $r-1 \geq 3$ so in particular $x_{r-1}, x_{r+1} \neq v_{k-1}$.  Thus $j_r+j_{r+1} \geq 1$ and so $\sum_{i=1}^{n-1} j_i\geq 2$ as desired.

Now we consider an almost distance maximizing radio labeling $c'$ of $S_{2k,s}$. As $c'$ is almost distance maximizing exactly one of the $n(e_i)$ values considered above is exactly one less. If this value is $n(e_{k-1})$, then all values for $n(e_i)$ would be even contradicting conclusion 2 of Lemma \ref{lem:distances}. Thus $n(e_{k-1})=2k-1$ in this case too so by conclusion 2 of Lemma \ref{lem:distances} if $x_{i}$ is in the set $\{v_k, ..., v_{2k-1}\}$, then the hypotheses of the claim are satisfied for the triple $\{x_{i-1}, x_i, x_{i+1}\}$. Therefore the above argument when $x_m=v_{2k-1}$ still holds and so $\sum_{i=1}^{n-1} j_i\geq 1$.

In conclusion, we have shown that if $c$ is distance maximizing then $\sum_{i=1}^{n-1} j_i\geq 2$ and if $c'$ is almost distance maximizing then $\sum_{i=1}^{n-1} j_i\geq 1$. In either case by Proposition \ref{prop:untelescope} we conclude that $rn(S_{2k,s}) \geq 2k^2 - 4k + 2s + 3.$ If $c$ is neither distance maximizing, nor almost distance maximizing then by Proposition \ref{prop:untelescope} it follows that $rn(S_{2k,s}) \geq 2k^2 - 4k + 2s + 3$ as $\sum_{i=1}^{n-1} j_i$ is always non-negative.

\end{proof}

\section{Radio number of all other diameter $n-2$ graphs}\label{sec:others}

In this section we will determine the radio number of all other diameter $n-2$ graphs.  We start  with some definitions.

\begin{definition}Let $n,s \in \mathbb{Z}$ where  $n \geq 4$ and $2 \leq s \leq n$. 
We define the graph $S^1_{n,s}$  with vertices $\{v_1,...,v_n\}$, and edges $\{v_i,v_{i+1} | i=1,2,...,n-2\}$ together with the edges $\{v_s,v_n\}$ and $\{v_{s-1},v_n\}$.  Without loss of generality we will always assume that $s \leq \lfloor\frac{n+1}{2}\rfloor$. See Figure \ref{fig:SpireGraphS_1}.
\end{definition}

\begin{figure}[h]
\begin{center}
\includegraphics[scale=.5]{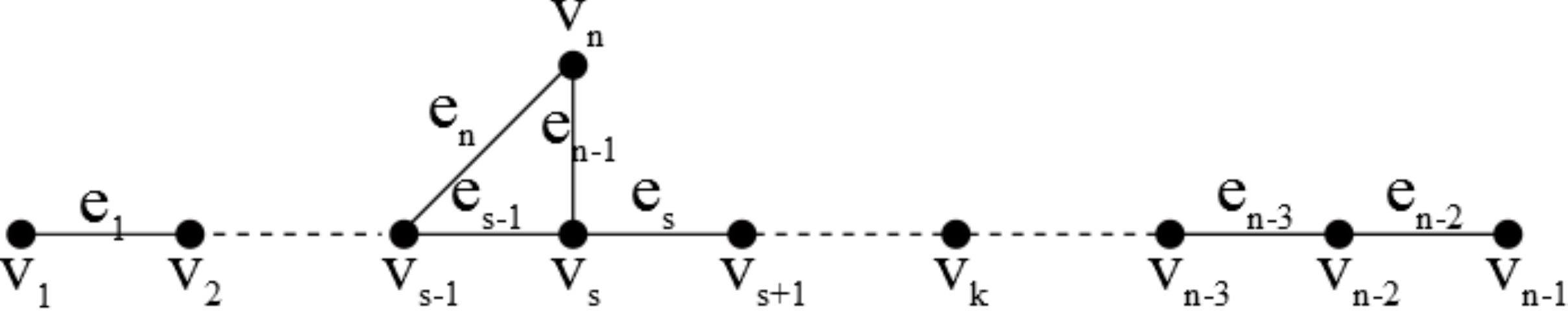}
\caption{$S^1_{n,s}$}
\label{fig:SpireGraphS_1}
\end{center}
\end{figure}

\begin{definition}Let $n,s \in \mathbb{Z}$ where  $n \geq 4$ and $3 \leq s \leq n$. 
We define the graph $S^{2}_{n,s}$  with vertices $\{v_1,...,v_n\}$, and edges $\{v_i,v_{i+1} | i=1,2,...,n-2\}$ together with the edges $\{v_s,v_n\}$ and $\{v_{s-2},v_n\}$.  Without loss of generality we will always assume that $s \leq \lfloor\frac{n+2}{2}\rfloor$. See Figure \ref{fig:SpireGraphS_2}.
\end{definition}

\begin{figure}[h]
\begin{center}
\includegraphics[scale=.5]{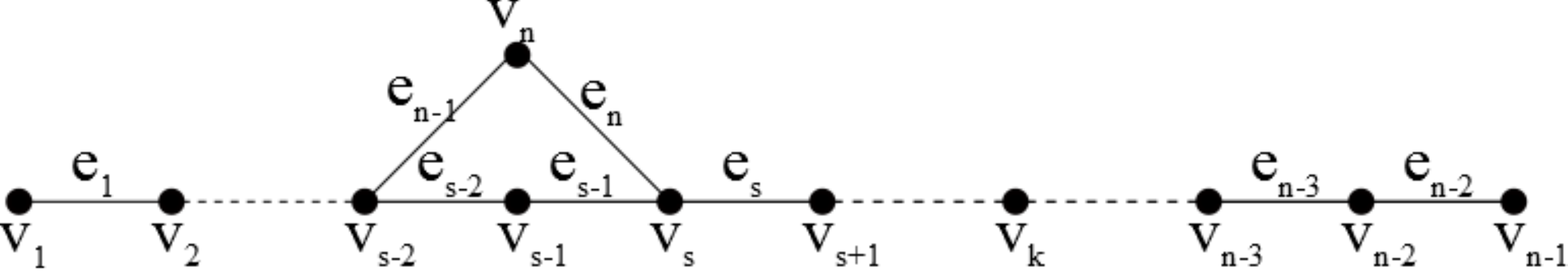}
\caption{$S^2_{n,s}$}
\label{fig:SpireGraphS_2}
\end{center}
\end{figure}

\begin{definition}Let $n,s \in \mathbb{Z}$ where  $n \geq 4$ and $3 \leq s \leq n$. 
We define the graph $S^{1,2}_{n,s}$  with vertices $\{v_1,...,v_n\}$, and edges $\{v_i,v_{i+1} | i=1,2,...,n-2\}$ together with the edges $\{v_s,v_n\}$, $\{v_{s-1},v_n\}$ and $\{v_{s-2},v_n\}$.  Without loss of generality we will always assume that $s \leq \lfloor\frac{n+2}{2}\rfloor$. See Figure \ref{fig:SpireGraphS_1,2}.
\end{definition}

\begin{figure}[h]
\begin{center}
\includegraphics[scale=.5]{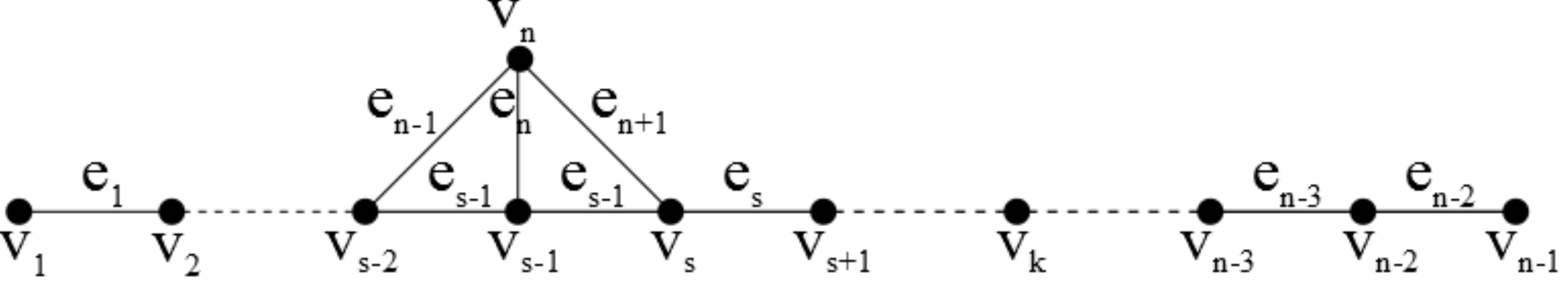}
\caption{$S^{1,2}_{n,s}$}
\label{fig:SpireGraphS_1,2}
\end{center}
\end{figure}

Note that these and spire graphs are all possible $n$-vertex, $n-2$-diameter graphs.  To determine the radio numbers of these graphs, we begin with an easy remark:

\begin{rmk} \label{rmk:edge removal}
If a connected graph $G'$  results from removing one or more edges from a connected graph $G$ and $diam(G') = diam(G)$, then $rn(G') \leq rn(G)$.
\end{rmk}

\begin{theorem}\label{thm:others}
For $2 \leq s \leq \lfloor\frac{n}{2}\rfloor$, $rn(S^*_{n,s})=rn(S_{n,s})$ where $rn(S^*_{n,s})$ is any one of $rn(S^{1}_{n,s})$, $rn(S^{2}_{n,s})$ or $rn(S^{1,2}_{n,s})$.

\end{theorem}

\begin{proof}
For $2 \leq s \leq \lfloor\frac{n}{2}\rfloor$ the graph $S_{n,s}$ results from removing an edge from either $S^{1}_{n,s}$ or $S^{2}_{n,s}$, both of which result from removing an edge from $S^{1,2}_{n,s}$.  Since all the graphs have diameter $n-2$, by Remark \ref{rmk:edge removal}
$$rn(S_{n,s}) \leq rn(S^{1}_{n,s}) \leq rn(S^{1,2}_{n,s}), \text{ and} $$ 
$$rn(S_{n,s}) \leq rn(S^{2}_{n,s}) \leq rn(S^{1,2}_{n,s}).$$

By the above discussion, we only need to show that $rn(S_{n,s}) \geq rn(S^*_{n,s}) $. We will do that by demonstrating that the radio labeling for $S_{n,s}$ given in Theorem \ref{thm:upper} induces a radio labeling for $S^*_{n,s}$ with the same span. Let $\{v_1,...,v_n\}$ be the vertices of $S_{n,s}$ and let $\{v^*_1,...,v^*_n\}$ be the vertices of $S^*_{n,s}$. Let $c^*:V(S^*_{n,s})\rightarrow \mathbb Z_+$ be given by $c^*(v^*_i)=c(v_i)$ where $c$ is the function in Theorem \ref{thm:upper} (for the corresponding case).

Notice that $d(v^*_i, v^*_j) = d(v_i, v_j)$ for all  $j>i$ except possibly when $j=n$ and $i \leq s-1$.  Thus to verify that $c^*$ is a radio labeling, we only need to verify the radio condition for the pairs $\{v^*_i, v^*_n\}$, where $i \leq s-1$.
 
{\bf Case 1:} $n=2k$ and $s\leq k-2$. 
  
	By Theorem \ref{thm:upper} we have that $c^*(v^*_n)=c(x_2)$ so we verify the radio condition for all pairs $\{x_i, x_2\}$.  Recall that we are assuming that $s\geq 2$ and so $k\geq 4$. By adding the entries in the $2^{nd}$, $3^{rd}$, and $4^{th}$ rows of the last column of Table \ref{Table 1}, we calculate that for all $i \geq 5$, $c^*(x_i)-c^*(x_2) \geq 3k+s-5 \geq 2k-1$.

	Thus regardless of the value of $s$, the radio condition is satisfied for all $i \geq 5$.  Note that $x_1$ corresponds to $v^*_k$ and $x_3$ corresponds to $v^*_{k+4}$. As $s\leq k-2$, $d(v_i, v_n)=d(v^*_i, v^*_n)$ for $i=k,k+4$ so the radio condition is satisfied for these pairs. Finally we consider the pair $\{x_4, x_2\}$.  Noting that $x_4$ corresponds to $v^*_5$, we have that $d(v_5, v_n)=d(v^*_5, v^*_n)$ if $s \leq 5$ and the radio condition is satisfied.  If $s \geq 6$, then by adding the entries in the $2^{nd}$ and $3^{rd}$ rows of the last column of Table \ref{Table 1}, we calculate that $c^*(x_4)-c^*(x_2) = 2k+s-6 \geq 2k+6-6= 2k$, and the radio condition is satisfied.  

{\bf   Case 2:} $n=2k$, and $s=k-1$ or $s=k$. 
   
As these cases are straightforward, we leave it to the reader to check them using Tables 3 and 4.
 
{\bf Case 3:} $n=2k+1$ and $2\leq s\leq k$. 

The reader can check these using Tables 5 and 6. \end{proof}

Theorem \ref{thm:others} leaves out only a few graphs with diameter $n-2$. The following theorem establishes the radio number in those cases:

\begin{theorem}

$rn(S^1_{2k+1,k+1})=2k^2 +1$.

$rn(S^{1,2}_{2k+1,k+1})=2k^2 +1$.

$rn(S^2_{2k+1,k+1})=2k^2$.

$rn(S^{1,2}_{2k,k+1})=2k^2 -2k+2$.

$rn(S^2_{2k,k+1})=2k^2 -2k+1$.

\end{theorem}
\begin{proof}
{\bf Case 1:} $S^1_{2k+1,k+1}$.
We first prove that $2k^2+1$ is an upper bound for $rn(S^1_{2k+1,k+1})$.  Order the vertices of $S^1_{2k+1,k+1}$ into three groups as follows:
\begin{center}
Group I: $v_k, v_{2k+1}$,\\
Group II: $v_{2k}, v_{k-1}, v_{2k-1}, v_{k-2}, ..., v_{k+2}, v_1$,\\
Group III: $v_{k+1}$.
\end{center}

Now, rename the vertices of $S^1_{2k+1,k+1}$ in the above ordering by $x_1, x_2, \ldots, x_n$.  This is the label order of the vertices of $S^1_{2k+1,k+1}$.\\

{\bf Claim:} The function $c$ defined in Table \ref{Table 8} is a radio labeling on $S^1_{2k+1,k+1}$.\\

\begin{table}
\centering
\begin{tabular}{|c|c|c|c|}
\hline
$x_i$ & Vertex Names & $d(x_i, x_{i+1})$ & $c(x_{i+1})-c(x_i)$\\
\hline
$x_1$ & $v_k$ & $1$ & $2k-1$ \\
$x_2$ & $v_{2k+1}$ & $k$ & $k$\\
\hline
$x_3$ & $v_{2k}$ & $k+1$ & $k-1$\\
$x_4$ & $v_{k-1}$ & $k$ & $k$\\
$x_5$ & $v_{2k-1}$ & $k+1$ & $k-1$\\
$x_6$ & $v_{k-2}$ & $k$ & $k$\\
$\vdots$ & $\vdots$& $\vdots$ & $\vdots$\\
$x_{n-4}$ & $v_{k+3}$ & $k+1$ & $k-1$\\
$x_{n-3}$ & $v_2$ & $k$ & $k$\\
$x_{n-2}$ & $v_{k+2}$ & $k+1$ & $k-1$\\
$x_{n-1}$ & $v_1$ & $k$ & $k$\\
\hline
$x_n$ & $v_{k+1}$ & n/a & n/a\\
\hline
\end{tabular}
\caption{}
\label{Table 8}
\end{table}
{\em Proof of claim:} We let the reader verify that the radio condition holds for all vertices $x_i, x_j \in V(S^1_{2k+1,k+1})$.  In this case, the diameter of $S^1_{2k+1,k+1}$ is $2k-1$ so for every $i,j$ with $j>i$, $d(x_i, x_j)+c(x_j)-c(x_i) \geq 2k$ must hold.\\

Letting $c(x_1)=1$, the largest number in the range of the radio labeling $c$ is then $c(x_n)$ and is therefore equal to the sum of the entries in the last column of Table \ref{Table 8} plus one.  We let the reader verify that $\rn(S^1_{2k+1,k+1}) \leq 2k^2+1$ as desired.\\

{\bf Claim:} $rn(S^1_{2k+1,k+1}) \geq 2k^2+1$.\\
\indent {\em Proof of claim:} We find a lower bound for $rn(S^1_{2k+1,k+1})$ by using Lemma \ref{lemma:distance} and determining $\max_p \sum_{i=1}^{n-1} d(x_i,x_{i+1})$.  For $1\leq i \leq 2k-2$ let $e_i$ be the edge between $v_i$ and $v_{i+1}$. Let $e_{2k}$ and $e_{2k+1}$ be the two edges incident to $v_{2k+1}$. We will use the terminology established in Lemma \ref{lem:distances}. Using the third conclusion of that lemma, it follows that

\begin{center}

$$N(e_i) \leq \begin{cases}

2i & \text{if } i\leq k-1, \\

2(2k-i)& \text{if }  k+1\leq i\leq 2k-1.  \\

\end{cases}
$$
\end{center}

Furthermore note that any path $P_j$ contains at most one of $e_k$, $e_{2k}$ and $e_{2k+1}$. As there are a total of $2k$ paths $P_j$, it follows that $n(e_k)+n(e_{2k})+n(e_{2k+1})\leq 2k$. Therefore $\max_p \sum_{i=1}^{n-1} d(x_i,x_{i+1}) \leq \sum_{i=1}^{2k+1} N(e_i)= 2k^2$, and Lemma \ref{lemma:distance} shows that $rn(S^1_{2k+1,k+1}) \geq 4k^2 + 1 -2k^2 = 2k^2 + 1$ as desired.

{\bf Case 2:} $S^{1,2}_{2k+1,k+1}$. 
Note that $S^1_{2k+1,k+1}$ results from removing an edge from $S^{1,2}_{2k+1,k+1}$ (and the graphs have the same diameter), so by Remark \ref{rmk:edge removal} and Case 1, $rn(S^1_{2k+1,k+1}) = 2k^2+1 \leq rn(S^{1,2}_{2k+1,k+1})$.  We leave it to the reader to verify that the same labeling in Table \ref{Table 8} is valid.

{\bf Case 3:} $S^2_{2k+1,k+1}$.
Notice that $S_{2k+1,k+1} = S_{2k+1,k}$ by symmetry.  Then since $S_{2k+1,k+1}$ results from removing an edge from $S^2_{2k+1,k+1}$ (and the graphs have the same diameter), we have by Remark \ref{rmk:edge removal}, Theorem \ref{thm:upper}, and Theorem \ref{thm:lower} that $rn(S_{2k+1,k+1})=rn(S_{2k+1,k}) = 2k^2 \leq rn(S^2_{2k+1,k+1})$.  We use the labeling of Table \ref{Table 8} making the change that $c(x_2) - c(x_1) = 2k-2$ since now $d(x_1,x_2) = 2$ to conclude that $rn(S^2_{2k+1,k+1}) \leq 2k^2$.

{\bf Case 4:} $S^{1,2}_{2k,k+1}$.
We first prove that $2k^2-2k+2$ is an upper bound for $rn(S^{1,2}_{2k,k+1})$.  Order the vertices of $S^{1,2}_{2k,k+1}$ into three groups as follows:
\begin{center}
Group I: $v_k, v_{2k}, v_{2k-1}$,\\
Group II: $v_1, v_{k+1}, v_2, v_{k+2}, ..., v_{k-2}, v_{2k-2}$,\\
Group III: $v_{k-1}$.
\end{center}

Now, rename the vertices of $S^{1,2}_{2k,k+1}$ in the above ordering by $x_1, x_2, \ldots, x_n$.  This is the label order of the vertices of $S^{1,2}_{2k,k+1}$.\\

{\bf Claim:} The function $c$ defined in Table \ref{Table 9} is a radio labeling on $S^{1,2}_{2k,k+1}$.\\

\begin{table}
\centering
\begin{tabular}{|c|c|c|c|}
\hline
$x_i$ & Vertex Names & $d(x_i, x_{i+1})$ & $c(x_{i+1})-c(x_i)$\\
\hline
$x_1$ & $v_k$ & $1$ & $2k-2$ \\
$x_2$ & $v_{2k}$ & $k-1$ & $k$\\
$x_3$ & $v_{2k-1}$ & $2k-2$ & $1$\\
\hline
$x_4$ & $v_1$ & $k$ & $k-1$\\
$x_5$ & $v_{k+1}$ & $k-1$ & $k$\\
$x_6$ & $v_2$ & $k$ & $k-1$\\
$x_7$ & $v_{k+2}$ & $k-1$ & $k$\\
$\vdots$ & $\vdots$& $\vdots$ & $\vdots$\\
$x_{n-4}$ & $v_{k-3}$ & $k$ & $k-1$\\
$x_{n-3}$ & $v_{2k-3}$ & $k-1$ & $k$\\
$x_{n-2}$ & $v_{k-2}$ & $k$ & $k-1$\\
$x_{n-1}$ & $v_{2k-2}$ & $k-1$ & $k$\\
\hline
$x_n$ & $v_{k-1}$ & n/a & n/a\\
\hline
\end{tabular}
\caption{}
\label{Table 9}
\end{table}
{\em Proof of claim:} We let the reader verify that the radio condition holds for all vertices $x_i, x_j \in V(S^{1,2}_{2k,k+1})$.  In this case, the diameter of $S^{1,2}_{2k,k+1}$ is $2k-2$ so for every $i,j$ with $j>i$, $d(x_i, x_j)+c(x_j)-c(x_i) \geq 2k-1$ must hold.\\

Letting $c(x_1)=1$, the largest number in the range of the radio labeling $c$ is then $c(x_n)$ and is therefore equal to the sum of the entries in the last column of Table \ref{Table 9} plus one.  We let the reader verify that $\rn(S^{1,2}_{2k,k+1}) \leq 2k^2-2k+2$ as desired.\\

{\bf Claim:} $rn(S^{1,2}_{2k,k+1}) \geq 2k^2-2k+2$.\\
\indent {\em Proof of claim:} We find a lower bound for $rn(S^{1,2}_{2k,k+1})$ by using Lemma \ref{lemma:distance} and determining $\max_p \sum_{i=1}^{n-1} d(x_i,x_{i+1})$.  For $1\leq i \leq 2k-2$ let $e_i$ be the edge between $v_i$ and $v_{i+1}$. Let $e_{2k-1}$, $e_{2k}$ and $e_{2k+1}$ be the three edges incident to $v_{2k}$ where $e_{2k-1}$ is incident to $v_{k-1}$, $e_{2k}$ is incident to $v_{k}$ and $e_{2k+1}$ is incident to $v_{k+1}$ (see Figure \ref{fig:SpireGraphS_1,2}). By the third conclusion of Lemma \ref{lem:distances} it follows that
\begin{center}

$$N(e_i) \leq  \begin{cases}

2i & \text{if } 1\leq i\leq k-2, \\

2(2k-1-i)& \text{if }  k+1\leq i\leq 2k-2.  \\

\end{cases}
$$
\end{center}

Furthermore by Lemma \ref{lem:multiple} it follows that $N(e_{k-1})+N(e_{2k-1})\leq 2(k-1)$ and $N(e_{k})+N(e_{2k+1})\leq 2(k-1)$. Finally $n(e_{2k})\leq 1$ as it is only contained in a path with endpoints $v_k$ and $v_{2k}$. Note that if all three of these inequalities are equalities, then $v_k$ and $v_{2k}$ correspond to $x_1$ and $x_{2k}$ by the first conclusion of Lemma \ref{lem:distances} as these are the only vertices for which the sum of the $n(e_i)$ for the incident edges may be odd. At the same time $v_k$ and $v_{2k}$ must correspond to $x_i$ and $x_{i+1}$ for some $i$ as $n(e_{2k})= 1$. This is a contradiction. Therefore $n(e_{k-1})+n(e_{k})+n(e_{2k-1})+n(e_{2k})+n(e_{2k+1})\leq 4(k-1)$.  Thus $\max_p \sum_{i=1}^{n-1} d(x_i,x_{i+1}) \leq 2k^2-2k$, and Lemma \ref{lemma:distance} shows that $rn(S^{1,2}_{2k,k+1}) \geq 4k^2 -4k +2 -2k^2 +2k = 2k^2 -2k+ 2$.

{\bf Case 5:} $S^2_{2k,k+1}$.
We use the labeling of Table \ref{Table 9} making the change that $c(x_2) - c(x_1) = 2k-3$ since now $d(x_1,x_2) = 2$ to conclude that $rn(S^2_{2k,k+1}) \leq 2k^2-2k+1$.
\indent  For $1\leq i \leq 2k-2$ let $e_i$ be the edge between $v_i$ and $v_{i+1}$. Let $e_{2k-1}$ and $e_{2k}$ be the  edges incident to $v_{2k}$ where $e_{2k-1}$ is incident to $v_{k-1}$, and $e_{2k}$ is incident to $v_{k+1}$. As in the previous case it follows that
\begin{center}

$$N(e_i) \leq  \begin{cases}

2i & \text{if } i\leq k-2, \\

2(2k-1-i)& \text{if }  k+1\leq i\leq 2k-2.  \\

\end{cases}
$$
\end{center}

Unlike in the previous case, here exactly one path may contain $e_{k-1}$ and $e_{2k-1}$ or it may contain $e_{k}$ and $e_{2k}$. This would be the path (if such a path exists) with endpoints $v_k$ and $v_{2k}$. Without loss of generality we can assume that this path contains $e_{k-1}$ and $e_{2k-1}$. Therefore in this case $n(e_{k-1})+n(e_{2k-1})\leq 2(k-1)+1$ and $n(e_{k})+n(e_{2k})\leq 2(k-1)$. Thus  $\max_p \sum_{i=1}^{n-1} d(x_i,x_{i+1}) \leq 2k^2-2k+1$, and Lemma \ref{lemma:distance} shows that $rn(S^{1,2}_{2k,k+1}) \geq 4k^2 -4k +2 -2k^2 +2k-1 = 2k^2 -2k+ 1$.

\end{proof}

\begin{center}{\bf Appendix:\\ Labelings of Graphs of order $n=2k$ with $k \leq 7$ and diameter $n-2$}\end{center}

The figures below give upper bounds for the radio number of spire graphs with $k\leq 7$ and $n=2k$ since these particular cases were not covered in Theorem \ref{thm:upper}.  These upper bounds match the lower bounds for these graphs found in Theorem \ref{thm:lower} to show that these bounds are the actual radio number of the graphs.

\begin{figure}[h!]
\begin{center}
\includegraphics[scale=.5]{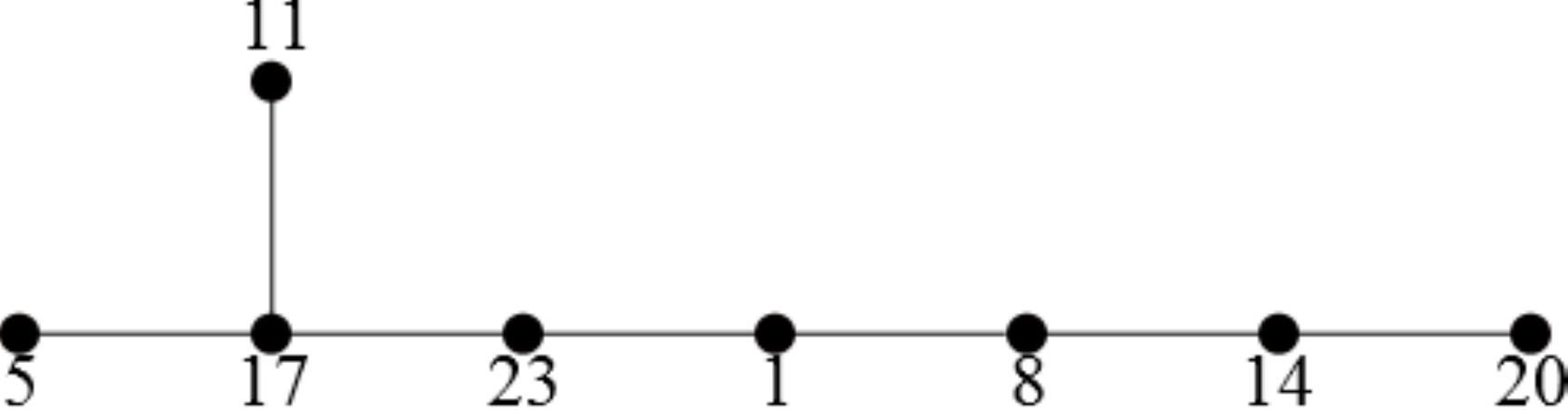}
\caption{$rn(S_{8,2})\leq 23$}
\label{fig:$S_{8,2}$}
\end{center}
\end{figure}

\begin{figure}[h!]
\begin{center}
\includegraphics[scale=.5]{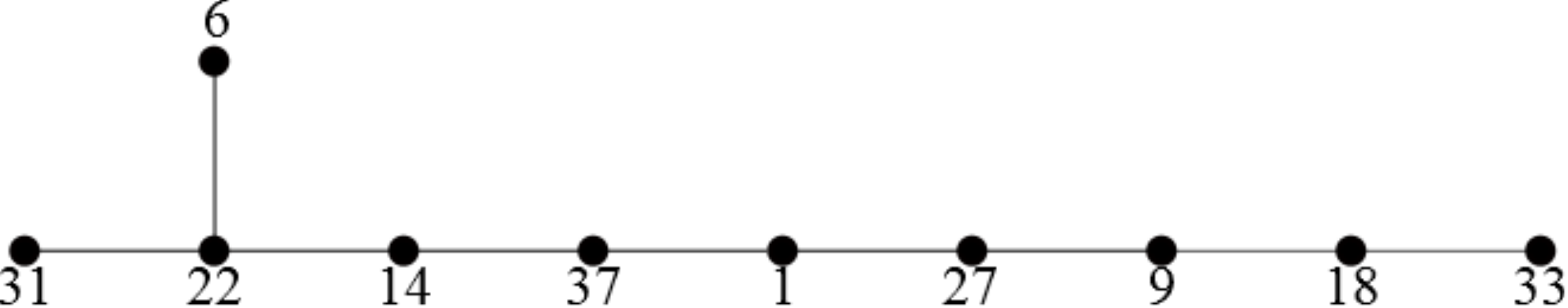}
\caption{$rn(S_{10,2}) \leq 37$}
\label{fig:$S_{10,2}$}
\end{center}
\end{figure}

\begin{figure}[h!]
\begin{center}
\includegraphics[scale=.5]{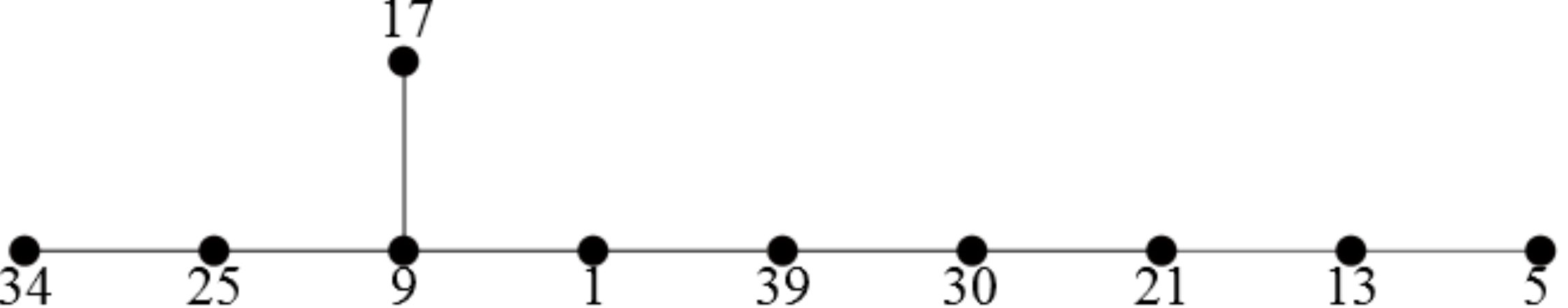}
\caption{$rn(S_{10,3}) \leq 39$}
\label{fig:$S_{10,3}$}
\end{center}
\end{figure}

\begin{figure}[h!]
\begin{center}
\includegraphics[scale=.5]{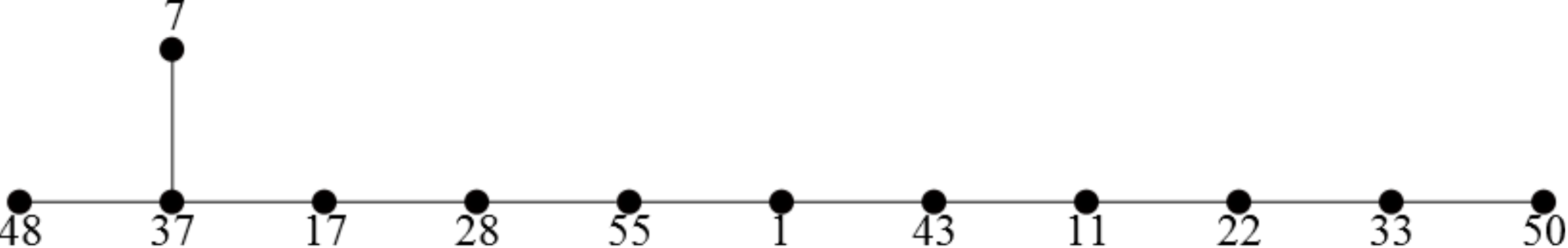}
\caption{$rn(S_{12,2}) \leq 55$}
\label{fig:$S_{12,2}$}
\end{center}
\end{figure}

\begin{figure}[h!]
\begin{center}
\includegraphics[scale=.5]{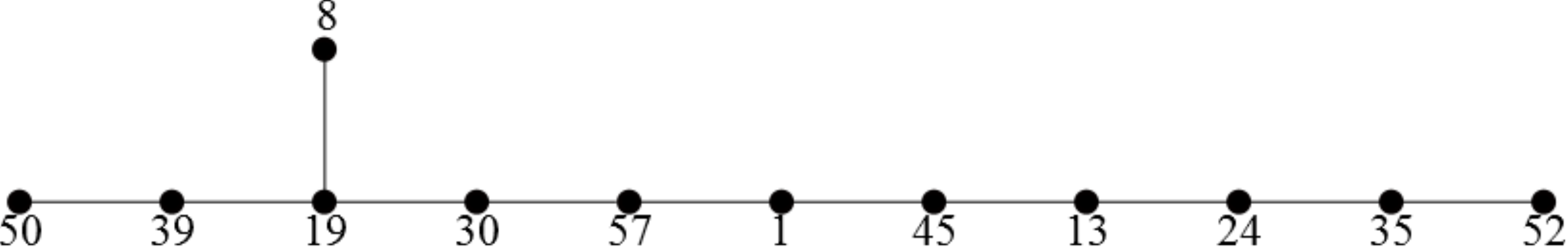}
\caption{$rn(S_{12,3}) \leq 57$}
\label{fig:$S_{12,3}$}
\end{center}
\end{figure}

\begin{figure}[h!]
\begin{center}
\includegraphics[scale=.5]{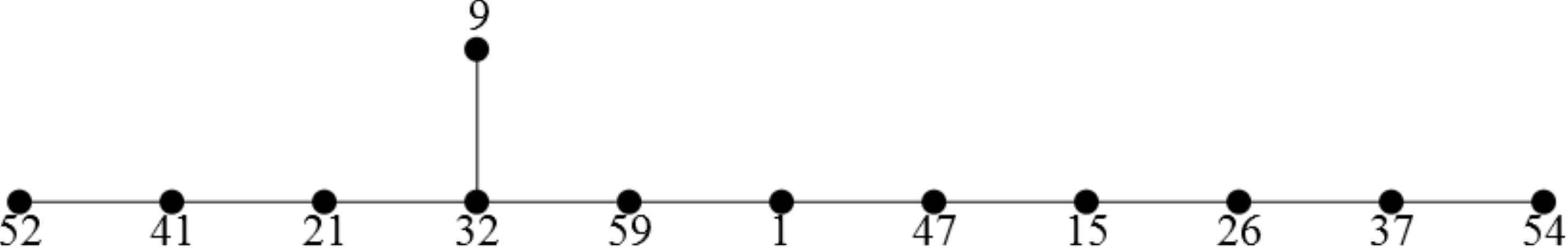}
\caption{$rn(S_{12,4}) \leq 59$}
\label{fig:$S_{12,4}$}
\end{center}
\end{figure}


\end{document}